\documentclass[11pt]{article}

\usepackage{a4}
\usepackage{amssymb}
\usepackage{amsmath}

\usepackage[latin1]{inputenc}

\usepackage{pstricks}
\usepackage[dvips,all]{xy}

\usepackage{srcltx}

\usepackage{algorithm}
\usepackage{algorithmic}

\newcommand{\ds}                   {\displaystyle}

\newcommand{\ADM}                  {{\textrm{ADM}}}

\newcommand{\qurtains}             {\hfill{QED}}

\ifx \theorem \undefined
\newtheorem{definition}{\vspace{1mm}Definition}[section]

\newtheorem{pth}[definition]{\vspace{1mm}Theorem}
\newtheorem{corollary}[definition]{\vspace{1mm}Corollary}
\newtheorem{prop}[definition]{\vspace{1mm}Proposition}

\newenvironment{proof}{\begin{trivlist}\item{\bf Proof:}}{\qurtains\end{trivlist}}

\newcommand{\nats}                  {{\mathbb N}}

\newcommand{\udl}[1]                {{\underline{#1}}}

\newcommand{\lverum}                {{\mathsf{t\!t}}}
\newcommand{\lfalsum}               {{\mathsf{f\!f}}}

\newcommand{\lnec}                  {{\mathop{\Box}}}
\newcommand{\lposs}                  {{\mathop{\lozenge}}}
\newcommand{\lneg}                  {{\mathop{\neg}}}

\newcommand{\limp}                  {\mathbin{\supset}}
\newcommand{\leqv}                  {\mathbin{\equiv}}
\newcommand{\lconj}                 {\mathbin{\wedge}}
\newcommand{\ldisj}                 {\mathbin{\vee}}

\newcommand{\sat}                    {\Vdash}
\newcommand{\ent}                    {\vDash}
\newcommand{\der}                    {\vdash}

\newcommand {\cL}                   {\mathcal{L}}
\newcommand {\cM}                   {\mathcal{M}}

\newcommand {\frA}                  {\mathfrak{A}}

\newcommand{\mverum}[1]             {\lverum^{(#1)}}

\newcommand {\Sigmamc}              {\Sigma_{\mc{12}}}
\newcommand {\Sigmamcm}[1]          {\Sigma_{{\mc{12}}{#1}}}
\newcommand {\Deltamc}              {\Delta_{\mc{12}}}
\newcommand {\cMmc}                 {\cM_{\mc{12}}}
\newcommand {\Lmc}                  {L_{\mc{12}}}

\newcommand {\satmc}                {\sat_{\mc{12}}}
\newcommand {\entmc}                {\ent_{\mc{12}}}
\newcommand {\dermc}                {\der_{\mc{12}}}

\newcommand{\mc}[1]                 {\mathop{\lceil#1\rceil}}
\newcommand{\bmc}[2]                {\left\lceil\arraycolsep=0.0pt \begin{array}{l}#1\\[-1mm]#2 \end{array}\right\rceil}
\newcommand{\bmcu}[2]                {\mathop{\bmc{#1}{#2}}}
\newcommand{\bmcb}[2]                {\mathbin{\bmc{#1}{#2}}}

\newcommand {\mysub}[3]		 		{#1|^{#2}_{#3}}
\newcommand {\suball}[2]		    {\mysub{#1}{}{#2}}

\newcommand {\HYP}					{\text{HYP}}

\newcommand {\MP}                   {\text{MP}}

\newcommand {\LFT}                  {\text{LFT}}

\newcommand {\cLFT}                 {\text{cLFT}}
\newcommand {\FX}                   {\text{FX}}
\newcommand {\GL}                   {\mathsf{GL}}

\newcommand {\SQT}                   {\mathsf{S4.3}}

\newcommand {\IPL}                  {\mathsf{IPL}}

\begin{document}

\title{Preservation of admissible rules\\
 when \\
combining logics}

\author{J.~Rasga\ \ C.~Sernadas\ \ A.~Sernadas\\[1mm]
{\scriptsize Dep.~Matemática, Instituto Superior Técnico, Universidade de Lisboa, Portugal}\\[-1mm]
{\scriptsize and}\\[-1mm]
{\scriptsize SQIG, Instituto de Telecomunicações, Lisboa, Portugal}\\
{\tiny \{joao.rasga,cristina.sernadas,amilcar.sernadas\}@tecnico.ulisboa.pt}}

\date{July 13, 2015}

\maketitle

\begin{abstract}
Admissible rules are shown to be conservatively preserved by the meet-combination of a wide class of logics.
A basis is obtained for the resulting logic from bases given for the component logics. 
Structural completeness and decidability of the set of admissible rules are also shown to be preserved, the latter with no penalty on the time complexity.
Examples are provided for the meet-combination of intermediate and modal logics.
\\[2mm]
\end{abstract}

\section{Introduction}\label{sec:intro}

The notion of admissible rule was proposed by Lorenzen~\cite{lor:55} when analysing definitional reflection and the inversion principle.\footnote{Recall that a rule is said to be {\it admissible} if every instantiation that makes the premises theorems also makes the conclusion a theorem.}
 Although every derivable rule is admissible, the converse does not always hold. 
 For instance, intermediate logics such as intuitionistic logic and the Gabbay-de Jongh logics \cite{gab:74} have admissible rules that are not derivable. 
One of the most well known examples is the Harrop rule~\cite{har:60} which is admissible but not derivable in intuitionistic logic.
This rule was  later shown in~\cite{pru:73} to be also admissible for all intermediate logics.
Many examples of non-derivable admissible rules appear in the context of modal logics~\cite{ryb:97}.
On the other hand, there are logics, like for example propositional logic  and the G\"odel-Dummett logic  that are structurally complete, meaning that every admissible rule is derivable~\cite{pru:73,ols:08}. 

An important research line on admissibility is concerned with finding a finite (or at least a recursive) basis for admissible rules. 
This problem was addressed for instance in~\cite{ryb:97}, using algebraic techniques based on quasi-identities in order to establish the non existence of a finite basis for the admissible rules of intuitionistic logic. 
It was further investigated for many logics including intermediate and intuitionistic logics~\cite{iem:01,iem:05}, Gabbay-de Jongh logics~\cite{iem:14}, modal logics~\cite{ryb:97}, many-valued logics~\cite{jer:10} and a paraconsistent logic~\cite{odi:13}. 
The work on unification~\cite{ghi:99,ghi:00} was central to some of the main results on this front, namely for showing in~\cite{iem:01,iem:05}  that  the Visser's rules constitute a non-finite recursive basis for intuitionist logic, for constructing in~\cite{jer:05} explicit bases for several normal modal logics, and more recently in~\cite{iem:14} for providing a basis for the admissible rules of the Gabbay-de Jongh logics.  
A systematic presentation of analytic proof systems for deriving admissible rules for intuitionistic logic and for a wide class of modal logics extending $\mathsf{K4}$ was presented in~\cite{iem:09}. 

The question of whether the set of admissible rules for a logic is decidable is another important research line. 
This question was first raised for intuitionistic logic by Friedman in~\cite{fri:75} (see problem 40) and given a positive answer by Rybakov in~\cite{ryb:84}. 
This positive result was extended to superintuitionistic logic and some modal logics in~\cite{bab:92,ryb:97,ghi:00,jer:05}). 
On the negative side, in~\cite{wol:08} it was shown that admissibility is undecidable for the basic modal logics 
$\mathsf{K}$ and $\mathsf{K4}$ extended with the universal modality.  
Complexity of the admissibility decision problem was investigated in~\cite{jer:07,jer:13,met:10}.

The significance of the problem of combining logics has meanwhile been recognised in many application domains of logic, namely in knowledge representation and in formal specification and verification of algorithms and protocols, since, in general, the need for working with several calculi at the same time is the rule rather than the exception. 
For instance, assuming that we have a logic for reasoning about time and a logic for reasoning about space we may want to express properties involving time and space. 
That is, one frequently needs to set-up theories with components in different logic systems or even better to work with theories in the combination of those logic systems. 
Several forms of combination have been proposed like, for example, fusion of modal logics (proposed in~\cite{tho:84}, see also~\cite{kra:wol:91,gab:03}), fibring (proposed in~\cite{gab:96}, see also~\cite{gab:99,acs:css:ccal:98a,zan:acs:css:99,css:jfr:wcarnielli:01}) and meet-combination (proposed in~\cite{acs:css:jfr:11b}). 
The topic of combination of logics raised some significant theoretical problems, such as the preservation of meta-properties like completeness~\cite{kra:wol:91,zan:acs:css:99,css:jfr:wcarnielli:01,acs:css:jfr:11b}, interpolation~\cite{kra:wol:91,wcarnielli:jfr:css:04,css:jfr:acs:12} and decidability~\cite{kra:wol:91}. 

Herein, we investigate admissibility preservation and related issues, like structural completeness, existence of bases, decidability and complexity of the set of admissible rules, in the context of combination of logics. 
We focus on meet-combination of logics because it is the weakest mechanism for combining logics in the sense that it minimizes the interaction between the components in the resulting logic. 
Section~\ref{sec:prems} includes a brief review of the relevant basic notions and results concerning meet-combination and admissibility, as well as examples concerning logics that are used throughout the paper.
We investigate the preservation of admissible rules by the meet-combination in Section~\ref{sec:admimeet} where we establish that a rule is admissible in the resulting logic whenever its projection to each component is admissible.
Moreover, we also show that this preservation is conservative under some mild assumptions.
In Section~\ref{sec:basespres} we address the problem of building a basis for admissible rules in the resulting logic from bases given for the component logics. 
In Section~\ref{sec:structcomp} we show that if the given logics are structurally complete then so is the logic resulting from their meet-combination, under some mild requirements. Moreover, we present a decision algorithm for admissibility in the resulting logic, using given decision algorithms for admissibility in the components. Finally, we analyse the time complexity of the algorithm.
In Section~\ref{sec:outlook}, we assess what was achieved and point out possible future developments.

\section{Preliminaries} \label{sec:prems}

By a {\it matrix logic} we mean a triple  $\cL=(\Sigma,\Delta,\cM)$ where:

\begin{itemize}

\item The {\it signature} $\Sigma$ is a family $\{\Sigma_n\}_{n\in\nats}$ with each 
$\Sigma_n$ being a set of $n$-ary language {\it constructors}.
Formulas are built as usual with the constructors and the {\it propositional or schema variables} in $\Xi=\{\xi_k \mid k\in\nats\}$. We use  $L(\Xi)$ for the set of all formulas.

\item The Hilbert {\it calculus} $\Delta$ is a set of rules.\footnote{By a (finitary) {\it rule} we mean a pair $(\{\alpha_1,\dots,\alpha_m\},\beta)$, denoted by $$\alpha_1\dots\alpha_m \; / \; \beta$$ 
where $\alpha_1,\dots,\alpha_m,\beta$ are formulas. Formulas $\alpha_1,\dots,\alpha_m$ are said to be the {\it premises} of the rule
and formula $\beta$ is said to be its {\it conclusion}.}
We write
$\Gamma \der \varphi$
for stating that $\varphi$ is derivable from set $\Gamma$ of formulas, that is, 
when there is a derivation of $\varphi$ from $\Gamma$. We write $\der \varphi$ whenever $\emptyset \der \varphi$, and say that $\varphi$ is a theorem.

\item The matrix {\it semantics} $\cM$ is a non empty class of matrices over $\Sigma$. Recall that a matrix $M$ is a pair $(\frA,D)$ where $\frA$ is an algebra over $\Sigma$ and $D$ is a non-empty subset of its carrier set $A$. The elements of $D$ are called distinguished values. 
We write
$\Gamma \ent \varphi$
for stating that, for each matrix $M=(\frA,D)$ and assignment $\rho:\Xi\to A$, 
if the denotation in $M$ and $\rho$ of each $\gamma\in \Gamma$ is a distinguished value, then so is the denotation of 
$\varphi$. Moreover, we write $\vDash\varphi$ whenever $\emptyset\vDash\varphi$.

\end{itemize}
Observe that $\cM$ is not necessarily the class of matrices canonically induced by $\Delta$,
as illustrated in due course.
Soundness and completeness are defined as expected.

A rule $\alpha_1\dots\alpha_m \; / \; \beta$  is said to be an  {\it admissible rule} of logic $\cL$ whenever 
$$\text{if $\der \sigma(\alpha_i)$ for each $i=1,\dots,m$ then $\der \sigma(\beta)$}$$
for every substitution $\sigma: \Xi \to L(\Xi)$.\footnote{There is an alternative definition of admissible rule but (see \cite{met:12}) it coincides with the one above when considering uni-conclusion rules.}
Given  a set of rules $\cal R$ and $\Gamma\cup\{\varphi\}\subseteq L(\Xi)$, 
$\varphi$ is said to be derivable by $\cal R$ from $\Gamma$, written
$$\Gamma \der^{\cal R} \varphi$$
whenever there is a sequence $\varphi_1\dots \varphi_m$ where $\varphi_m$ is $\varphi$ and for each $i=1,\dots,m$ either
$\varphi_i \in \Gamma$ or $(\{\varphi_{i_1},\dots, \varphi_{i_k}\}, \varphi_i)$ is a substitution instance of some rule in ${\cal R} \cup \Delta$, for some $i_1,\ldots,i_k$ less than $i$. A set of rules $\cal R$ constitutes a {\it basis} for some other set of rules $\cal R'$ if for every rule
$\alpha_1\dots\alpha_m \; / \; \beta$ in $\cal R'$ we have $\alpha_1\dots\alpha_m \der^{\cal R} \beta$. 
Logic $\cL$ is said to be \emph{structurally complete} whenever each admissible rule  is derivable. 


For defining meet-combination we need some assumptions on the logics at hand.  
We assume that $\lverum$ and $\lfalsum$ are in $\Sigma_0$ and are such that
$\der \lverum$ and $\lfalsum \der \varphi$ for every formula $\varphi$. 
Moreover, we assume that no matrix satisfies $\lfalsum$ and all the matrices satisfy $\lverum$.
Finally, for each $n \geq 1$, we assume that $\mverum{n}$ is in $\Sigma_n$ and is such that 
$\mverum{n}(\varphi_1,\dots,\varphi_n)$ is equivalent to $\lverum$.\footnote{In most logics such constructors could be introduced as abbreviations, as it is the case of all logics used in this paper.}

Given matrix logics
$\cL_1=(\Sigma_1,\Delta_1,\cM_1)$ and $\cL_2=(\Sigma_2,\Delta_2,\cM_2)$, 
their  {\it meet-combination} is the logic 
$$\mc{\cL_1\cL_2}=(\Sigmamc,\Deltamc,\cMmc)$$ 
where $\Sigmamc$, $\Deltamc$ and $\cMmc$ are as follows.

The signature $\Sigmamc$ is
such that, for each $n \in \nats$,
$$\Sigmamcm{n} = \{\mc{c_1c_2} \mid c_1 \in \Sigma_{1n},c_2\in\Sigma_{2n}\}.$$
The constructor $\mc{c_1c_2}$ is said to be the {\it meet-combination} of $c_1$ and $c_2$. In the sequel we may use
$$\bmc{c_1}{c_2}$$
instead of $\mc{c_1c_2}$ for the sake of readability. 
As expected, we use $\Lmc(\Xi)$ for denoting the set of  all formulas over $\Sigmamc$. Observe that we look at signature $\Sigmamc$ as an enrichment of $\Sigma_1$ via the embedding 
$\eta_1:c_1 \mapsto \mc{c_1\mverum{n}_2}$ for each $c_1 \in \Sigma_{1n}$
and similarly for $\Sigma_2$.
For the sake of lightness of notation, in the context of $\Sigmamc$, from now on, we may write 
$c_1$ for $\mc{c_1\mverum{n}_2}$ when $c_1\in\Sigma_{1n}$
and
$c_2$ for $\mc{\mverum{n}_1c_2}$ when $c_2\in\Sigma_{2n}$.
In this vein, for $k=1,2$, we look at $L_k(\Xi)$ as a subset of $\Lmc(\Xi)$.
Given a formula $\varphi$ over $\Sigmamc$ and $k\in\{1,2\}$, we denote by 
$$\suball{\varphi}{k}$$
the formula obtained from $\varphi$ by replacing every occurrence of each combined constructor by its $k$-th component. Such a formula is called the {\it projection} of $\varphi$ to $k$.

We need some notation before introducing the Hilbert calculus for the meet-combination.
Given a logic $\cL$ and a rule $r=\alpha_1\ldots\alpha_m\;/\;\beta$, the \emph{tagging of $r$ over} $\cL$, denoted by
$$\overline{r},$$
is the set of rules consisting of:
\begin{itemize}
	\item $r$ if $r$ is a {\it non-liberal rule} (that is, a rule where the conclusion is not a schema variable);
	\item for each $c \in \Sigma_{n}$ and $n\in\nats$, the rule
	$$\rho_{r,c}(\alpha_1)\quad\dots\quad\rho_{r,c}(\alpha_m) \; / \;
	    \rho_{r,c}(\beta)$$
where the substitution $\rho_{r,c}$ is such that $\rho_{r,c}(\xi)=\xi$ if $\xi$ is not $\beta$, and $\rho_{r,c}(\beta)=c(\xi_{j+1},\dots,\xi_{j+n})$ with $j$ being the maximum of the indexes of the schema variables occurring in $r$, if $r$ is a {\it liberal rule} (that is, a rule where the conclusion is a schema variable).
\end{itemize}
Moreover, given a set of rules $\mathcal{R}$ of $\cL$, we denote by
$\overline{\mathcal{R}}$
the set $$\ds\bigcup_{r\in\mathcal{R}}\overline{r}$$
of all tagged rules of $\mathcal{R}$ over $\cL$.

The calculus $\Deltamc$ is composed of the tagged version over $\cL_1$ of the rules inherited from $\Delta_1$ (via the implicit embedding $\eta_1$) and the tagged version over $\cL_2$ of the rules inherited from $\Delta_2$ (via the implicit embedding $\eta_2$), plus the rules imposing that each combined connective enjoys the common properties of its components and the rules for propagating falsum. More precisely, $\Deltamc$ contains the following rules:
\begin{itemize}
	
\item the {\it inherited rules} in $\overline{\Delta}_k$, for $k=1,2$;
	
\item the {\it lifting rule} (in short $\LFT$)
	$$\suball{\varphi}{1} \  \suball{\varphi}{2} \; / \; \varphi,$$
	for each formula $\varphi\in\Lmc(\Xi)$;
	
\item the {\it co-lifting rule} (in short $\cLFT$)
	$$\varphi \; / \; \suball{\varphi}{k},$$
	for each formula $\varphi\in\Lmc(\Xi)$ and $k=1,2$;

\item the {\it falsum propagation rules} (in short $\FX$) of the form 
		$$\displaystyle \lfalsum_1\; / \; \lfalsum_2 \quad \text{and} \quad \lfalsum_2\; / \;\lfalsum_1.$$
	
\end{itemize}

At first sight one might be tempted to include in $\Deltamc$ every rule in $\Delta_1\cup\Delta_2$.
For instance, if modus ponens ($\MP$) is a rule in $\Delta_1$ one would expect to find in $\Deltamc$ the rule
$\xi_1 \; \; (\xi_1 \limp_1 \xi_2)\; / \;\xi_2.$
However, as discussed in~\cite{acs:css:jfr:11b}, this rule is not sound. Instead, we tag such a liberal rule, including in $\Deltamc$,  the $c$-tagged modus ponens rule 
$\xi_1\; \; (\xi_1 \limp_1 c(\xi_3,\dots,\xi_{2+n})) \; / \;
       c(\xi_3,\dots,\xi_{2+n})$
       for each $c\in\Sigma_{1n}$ and $n\in\nats$.

The lifting rule is motivated by the idea that $\mc{c_1c_2}$ inherits the common properties of $c_1$ and $c_2$.
The co-lifting rule is motivated by the idea that $\mc{c_1c_2}$ should enjoy only the common properties of $c_1$ and $c_2$. 

The semantics $\cMmc$ is the class of product matrices
$$\{M_1 \times M_2 \mid M_1\in\cM_1 \; \text{and} \; M_2\in\cM_2\}$$
over $\Sigmamc$ such that each
$M_1 \times M_2 = (\frA_1 \times \frA_2,D_1 \times D_2)$
where
$$\frA_1 \times \frA_2 = (A_1 \times A_2,\{\udl{\mc{c_1c_2}}:(A_1 \times A_2)^n\to A_1 \times A_2 \mid 
                              \mc{c_1c_2}\in\Sigmamcm{n}\}_{n\in\nats})$$
with
$\udl{\mc{c_1c_2}}((a_1,b_1),\dots,(a_n,b_n)) =
	           (\udl{c_1}(a_1,\dots,a_n),\udl{c_2}(b_1,\dots,b_n)).$

Observe that, as shown in~\cite{acs:css:jfr:11b}, the meet-combination $\mc{\cL_1\cL_2}$ of two sound and concretely complete (that is, with respect to formulas without schema variables) matrix logics $\cL_1$ and $\cL_2$ provides an axiomatisation of the product of their matrix semantics since it preserves soundness and concretely completeness. Observe also that the embeddings $\eta_1$ and $\eta_2$ are conservative, as established in~\cite{acs:css:jfr:11b}.

\subsection*{Examples}

A matrix $M$  for intuitionist logic, referred to as $\IPL$ in the sequel, is   an Heyting algebra $\frA$ and the set of distinguished values is $D=\{\top\}$. Observe that $\IPL$ is not structurally complete since, for instance, the Harrop rule, 
$$ (\lneg \xi_1) \limp (\xi_2 \ldisj \xi_3)\, /\, ((\lneg \xi_1)\limp \xi_2) \ldisj ((\lneg \xi_1)\limp \xi_3)$$
is admissible but not derivable. A  basis for $\IPL$ is composed by the Visser's rules~\cite{iem:05}: 
$$\bigg\{\left(\bigwedge_{i=1} ^n(\xi_i \limp \xi'_i) \limp \xi_{n+1} \ldisj \xi_{n+2}\right) \ldisj \xi''\; {\bigg /} \; 
\bigvee_{j=1}^{n+2} \left(\bigwedge_{i=1} ^n(\xi_i \limp \xi'_i) \limp \xi_{j}\right) \ldisj \xi'': n \geq 1\bigg\}.$$
A matrix $M$ for modal logic $\SQT$ is an algebra $\frA_{(W,R)}$ induced by a Kripke frame $(W,R)$ where  
$R$ is reflexive, transitive and connected, i.e.~the carrier set $A$ is $\wp W$, $\bar \lneg$ is such that $\bar \lneg(U)= W \setminus U$, $\limp$ is such that 
$\bar \limp(U_1,U_2)= (W \setminus U_1) \cup U_2$, $\bar \lnec$ is such that
$\bar \lnec U=\{w \in W: \text{ if }  w R u \text{ then } u \in U\}$ and $D=\{W\}$. Again $\SQT$ is not structurally complete.
Observe that  the rule 
$$(\lnec \xi) \limp \xi \,/\, \xi$$ is non-admissible rule in $\SQT$.
A basis for $\SQT$ (see~\cite{ryb:97}) is a singleton composed by the admissible rule
$$(\lposs \xi)\lconj (\lposs \lneg \xi) \; / \; \lfalsum.$$

A matrix $M$ for modal logic $\GL$, after  G\"odel and L\"ob,  is an algebra $\frA_{(W,R)}$ induced by a Kripke frame $(W,R)$ where $R$ is transitive, finite and irreflexive, i.e.~the carrier set 
$A=\wp W$, $\bar \lneg$ is such that $\bar \lneg(U)= W \setminus U$, $\limp$ is such that 
$\bar \limp(U_1,U_2)= (W \setminus U_1) \cup U_2$, $\bar \lnec$ is such that
$\bar \lnec U=\{w \in W: \text{ if }  w R u \text{ then } u \in U\}$ and $D=\{W\}$. Again $\GL$ is not structurally complete.
A basis for $\GL$ (see~\cite{jer:05}) is the following set:
$$\bigg\{\lnec \left(\lnec \xi' \limp \bigvee_{i=1}^n \lnec \xi_i\right) \ldisj \lnec \xi'' \; / \; \bigvee_{i=1}^n \lnec (\xi' \lconj \lnec \xi' \limp \xi_i)\ldisj \xi'': n \geq 1\bigg\}.$$

Observe that the admissibility problem is decidable for all these logics. Moreover, this problem is co-NEXP-complete for $\IPL$ and $\GL$~\cite{jer:13}. The complexity of the problem for $\SQT$ is known to be co-NP-complete~\cite{jer:05}. 

As a first example of meet-combination, consider $\mc{\IPL\,\GL}$ where, concerning for instance negation, we find three variants: $\lneg_\IPL$, $\lneg_\GL$ and $\mc{\lneg_\IPL \lneg_\GL}$. Given the conservative nature of the embeddings of $\IPL$ and $\GL$ in $\mc{\IPL\,\GL}$, the negation $\lneg_\GL$ behaves as in $\GL$ in the image of $\GL$ in $\mc{\IPL\,\GL}$ and the negation $\lneg_\IPL$ behaves as in $\IPL$ in the image of $\IPL$ in $\mc{\IPL\,\GL}$. For instance, 
$$\der_{\mc{\IPL\,\GL}} \xi \ldisj_\GL ({\lneg}_{\GL} \xi) \; \text{ and  } \; \not\der_{\mc{\IPL\,\GL}} \xi \ldisj_\IPL ({\lneg}_{\IPL} \xi).$$
Concerning $\mc{\lneg_\IPL \lneg_\GL}$, as expected, it behaves intuitionistically since it inherits only the properties common to both negations.
So, for example
$$
\not \der_{\mc{\IPL\,\GL}}
\left(
\bmcu{\lneg_\IPL}{\lneg_\GL}
\bmcu{\lneg_\IPL}{\lneg_\GL}
\xi\right) 
\bmcb{\limp_\IPL}{\limp_\GL}
 \xi.
$$

\section{Admissible rules}\label{sec:admimeet}

In this section, we concentrate on the conservative preservation of admissible rules by the meet-combination. The preservation results assume that the component logics are sound. The conservativeness requires additional properties on the original logics.
Anyway the additional requirements are fulfilled by a large class of logics as we shall see below. 

\subsection{Preservation}
Before showing that admissible rules are preserved by meet-combination, we prove some relevant lemmas.~The following result establishes a relationship between substitution in the logic resulting from the meet-combination and substitution in each of the component logics. In the sequel,
given a substitution $\sigma$, we denote by $\sigma(\varphi)$ the formula that results from $\varphi$ by replacing each schema variable $\xi$ by $\sigma(\xi)$.

\begin{prop}\em\label{prop:subsprojs}
Let $\cL_1$ and $\cL_2$ be logics, $k$ in $\{1,2\}$, and $\rho: \Xi \to \Lmc(\Xi)$ and $\rho_k:\Xi \to L_k(\Xi)$ substitutions such that $\rho_k(\xi)=\suball{\rho(\xi)}{k}$ for each $\xi$ in $\Xi$.  Then
$$\rho_k(\suball{\psi}{k})=\suball{\rho(\psi)}{k} \quad \text{ for every }\psi \in\Lmc(\Xi).$$
\end{prop}
We omit the proof of Proposition~\ref{prop:subsprojs} since it follows by a straightforward induction on $\psi$. The following result establishes 
that the logic resulting from the meet-combination of sound logic is consistent.

\begin{prop}\em\label{prop:thmfalsum}
Let $\cL_1$ and $\cL_2$ be sound logics. 
Then, $\not \dermc \lfalsum_k$ for each $k=1,2$ and $\not\dermc\xi$ for each $\xi$ in $\Xi$. 
\end{prop}
\begin{proof}
We start by proving that $\not \dermc \lfalsum_k$ for each $k=1,2$. Assume by contradiction that $\dermc \lfalsum_1$. By soundness of the meet-combination $\entmc \lfalsum_1$.  Then, $\cMmc$ is empty.  Indeed, suppose by contradiction,  that $M_1 \times M_2 \in \cMmc$. Then $M_1 \times M_2 \satmc \lfalsum_1$ and so $M_1 \sat_1 \lfalsum_1$ contradicting the definition of $\lfalsum_1$. 
Hence either $\cM_1$ is empty or $\cM_2$ is empty. But, this contradicts the definition of matrix that assumes that 
$\cM_1 \neq \emptyset$ and $\cM_2 \neq \emptyset$. The proof that $\not \dermc \lfalsum_2$ is similar so we omit it.\\[2mm]
 The other assertion follows immediately taking into account that all the axioms in $\Deltamc$ are tagged.
\end{proof}

The following result relates derivations of theorems in the combined logic with derivations in each component logic.

\begin{prop}\em\label{prop:reflectdersub}
Let $\cL_1$ and $\cL_2$ be sound  logics and $\psi$ a formula in $\Lmc(\Xi)$. Assume that $\dermc \psi$. Then 
$\der_1\suball{\psi}{1}$ and $\der_2\suball{\psi}{2}$.
\end{prop}
\begin{proof} Let $\psi_1 \dots \psi_n$ be a derivation of $\psi$ in $\mc{\cL_1 \cL_2}$. Observe that $\psi$ and each $\psi_i$ for $i=1,\ldots,n-1$ are not schema variables by Proposition~\ref{prop:thmfalsum}. We prove the result by induction on $n$:\\[2mm]
Base.~The formula $\psi$ is an instance of an axiom $\alpha'$ in $\mc{\cL_1\cL_2}$  inherited from an axiom $\alpha$ either in $L_1(\Xi)$ or in $L_2(\Xi)$, by a substitution $\rho': \Xi \to \Lmc(\Xi)$. We now show that $\der_1\suball{\psi}{1}$ (we omit the proof of $\der_2\suball{\psi}{2}$ since it follows similarly). Consider two cases:\\[1mm]
(a)~$\alpha$ is in $L_1(\Xi)$. Let $\rho_1$ be a substitution over $L_1(\Xi)$ such that $\alpha'$ is $\rho_1(\alpha)$, and let $\rho'_1$ be a substitution over $L_1(\Xi)$ such that $\rho'_1(\xi)=\suball{\rho'(\xi)}{1}$ for every $\xi \in \Xi$. Then 
$\suball{\psi}{1}=\suball{\rho'(\rho_1(\alpha))}{1}=\rho'_1(\rho_1(\alpha))$, by Proposition~\ref{prop:subsprojs}. Hence, $\suball{\psi}{1}$ is an instance of $\alpha$ by $\rho'_1\circ\rho_1$, and, so, $\der_1\suball{\psi}{1}$.\\[1mm]
(b)~$\alpha$ is in $L_2(\Xi)$. Then, the head constructor of $\alpha'$ and, so, of $\psi$ is from $L_2(\Xi)$. Hence, the head constructor of $\suball{\psi}{1}$ is of the form $\mverum{n}$ for some $n$, and, so, $\der_1\suball{\psi}{1}$.\\[1mm]
Step. There are several cases to consider:\\[1mm]
(1)~$\psi$ results from $\psi_{i_1} \dots \psi_{i_m}$ by a rule 
$r'$ in $\mc{\cL_1\cL_2}$ inherited from a rule $r=\alpha_1\dots\alpha_m\;/\;\beta$ either in $\cL_1$ or in $\cL_2$, and by a substitution $\rho': \Xi \to \Lmc(\Xi)$. We now show that $\der_1\suball{\psi}{1}$ (we omit the proof of $\der_2\suball{\psi}{2}$ since it follows similarly). Consider two cases:\\[1mm]
(a)~$r$ is in $\cL_1$. Let $\rho'_1$ be a substitution over $L_1(\Xi)$ such that $\rho'_1(\xi)=\suball{\rho'(\xi)}{1}$ for every $\xi \in \Xi$ and $\rho_1$ a substitution over $L_1(\Xi)$ such that $r'$ is the rule $\rho_1(\alpha_1)\dots\rho_1(\alpha_m)\;/\;\rho_1(\beta)$. Hence, for $j=1,\dots,m$,
$$\der_1\suball{\rho'(\rho_1(\alpha_j))}{1}$$
by induction hypothesis, and, so,
$$\der_1(\rho'_1\circ\rho_1)(\alpha_j)$$
since $(\rho'_1\circ\rho_1)(\alpha_j)=\suball{\rho'(\rho_1(\alpha_j)}{1}$ by Proposition~\ref{prop:subsprojs}. Therefore, by rule $r$,
$$\der_1 (\rho'_1\circ\rho_1)(\beta),$$
and, so, the thesis follows  since $(\rho'_1\circ\rho_1)(\beta)=\suball{\rho'(\rho_1(\beta))}{1} = \suball{\psi}{1}$, by Proposition~\ref{prop:subsprojs}.\\[1mm]
(b)~$r$ is in $\cL_2$. Then, the head constructor of $\psi$ is in $\Sigma_2$. Hence, the head constructor of $\suball{\psi}{1}$ is of the form $\mverum{n}$ for some $n$, and, so, $\der_1\suball{\psi}{1}$.\\[1mm]
(2)~$\psi$ is obtained from $\suball{\psi}{1}$ and $\suball{\psi}{2}$ by rule $\LFT$. We now show that $\der_1\suball{\psi}{1}$ (we omit the proof of $\der_2\suball{\psi}{2}$ since it follows similarly). Observe that, by the induction hypothesis,
$\der_1 \suball{(\suball{\psi}{1})}{1}$. So, the thesis follows since
$\suball{(\suball{\psi}{1})}{1}$ is $\suball{\psi}{1}$.\\[2mm]
(3)~$\psi$ is $\suball{\psi_j}{1}$ and is obtained from $\psi_j$ using rule $\cLFT$. Then
$$\suball{\psi}{1}=\suball{\suball{\psi_j}{1}}{1}=\suball{\psi_j}{1}=\psi$$
Observe that, by the induction hypothesis, $\der_1\suball{\psi_j}{1}$, that is, $\der_1\suball{\psi}{1}$ as we wanted to show. On the other hand
$$\suball{\psi}{2}=\suball{\suball{\psi_j}{1}}{2}=\mverum{n}(\psi_1,\ldots,\psi_n)$$
for some $n\in\nats$ and formulas $\psi_1,\ldots,\psi_n$. Hence, $\der_2\suball{\psi}{2}$ as we wanted to show. 
The proof when $\psi$ is $\suball{\psi_j}{2}$  is similar.\\[2mm]
(4)~$\psi$ is $\lfalsum_1$ and is obtained using rule $\FX$. This case is not possible due to Proposition~\ref{prop:thmfalsum}.
Similarly when $\psi$ is $\lfalsum_2$.
\end{proof}

The following result asserts that a rule is admissible in the logic resulting from a meet-combination whenever both of its projections are admissible rules in the component logics.

\begin{pth}\label{th:preseradmissibambososlados} \em
Let $\cL_1$ and $\cL_2$ be sound  logics and $\alpha_1,\ldots,\alpha_m,\beta$ formulas of $\Lmc$ such that $\suball{\alpha_1}{1}\ldots\suball{\alpha_m}{1}\; / \; \suball{\beta}{1}$ and $\suball{\alpha_1}{2}\ldots\suball{\alpha_m}{2}\; / \; \suball{\beta}{2}$ are admissible rules of $\cL_1$ and $\cL_2$, respectively. Then, $$\alpha_1\ldots\alpha_m\;/\;\beta$$ is an admissible rule of $\mc{\cL_1\cL_2}$. 
\end{pth}
\begin{proof}
Let $\sigma$ be a substitution over $\mc{\cL_1\cL_2}$ such that
$$\dermc\sigma(\alpha_1)\quad\ldots\quad\dermc\sigma(\alpha_m).$$
Then, by Proposition~\ref{prop:reflectdersub}, 
$$\der_1\suball{\sigma(\alpha_1)}{1}\quad\ldots\quad\der_1\suball{\sigma(\alpha_m)}{1}$$
and
$$\der_2\suball{\sigma(\alpha_1)}{2}\quad\ldots\quad\der_2\suball{\sigma(\alpha_m)}{2}.$$
Let $\sigma_1$ and $\sigma_2$ be substitutions over $L_1(\Xi)$ and $L_2(\Xi)$ respectively, such that $\sigma_1(\xi)=\suball{\sigma(\xi)}{1}$ and $\sigma_2(\xi)=\suball{\sigma(\xi)}{2}$ for every schema variable $\xi$. Then,
$$\der_1\sigma_1(\suball{\alpha_1}{1})\quad\ldots\quad\der_1\sigma_1(\suball{\alpha_m}{1})$$
and
$$\der_2\sigma_2(\suball{\alpha_1}{2})\quad\ldots\quad\der_2\sigma_2(\suball{\alpha_m}{2})$$
by Proposition~\ref{prop:subsprojs}, and, so,
$$\der_1\sigma_1(\suball{\beta}{1})$$
and
$$\der_2\sigma_2(\suball{\beta}{2})$$
since $\suball{\alpha_1}{1}\ldots\suball{\alpha_m}{1}\; / \; \suball{\beta}{1}$ and $\suball{\alpha_1}{2}\ldots\suball{\alpha_m}{2}\; / \; \suball{\beta}{2}$ are admissible rules of $\cL_1$ and $\cL_2$, respectively. So,
$$\dermc\sigma_1(\suball{\beta}{1})$$
and
$$\dermc\sigma_2(\suball{\beta}{2})$$
taking into account that the logic resulting from the meet-combination is an extension of the component logics. Hence,
$$\dermc\suball{\sigma(\beta)}{1}$$
and
$$\dermc\suball{\sigma(\beta)}{2},$$
and, so,
$$\dermc\sigma(\beta)$$
by co-lifting, as we wanted to show.
\end{proof}

Taking into account that, in the meet-combination, the formula resulting from the projection to the other component logic of the conclusion of a tagged rule  is always a theorem of that logic, we can establish immediately the following corollary, capitalising on Theorem~\ref{th:preseradmissibambososlados}.

\begin{corollary}\label{cor:preseradmiss} \em
Let $\cL_1$ and $\cL_2$ be sound logics and $r$ an admissible rule of $\cL_k$, where $k \in \{1,2\}$. Then, each rule in $\overline{r}$ is an admissible rule of $\mc{\cL_1\cL_2}$. 
\end{corollary}

The following result states that a rule is admissible in the logic resulting from the meet-combination whenever its  projection to one of the components
is vacuously admissible. Moreover, when that is the case, its conclusion can be any formula.

\begin{pth}\label{th:preseradmissfalsum} \em
Let $\cL_1$ and $\cL_2$ be sound logics, $\alpha_1,\ldots,\alpha_m,\beta$ formulas of $\Lmc$, and $k$ in $\{1,2\}$. Then,
$$\suball{\alpha_1}{k}\ldots\suball{\alpha_m}{k}\; / \; \lfalsum_k\text{ is an admissible rule of }\cL_k$$
$$\Rightarrow$$
$$\alpha_1\ldots\alpha_m\; / \;\beta\text{ is an admissible rule of }\mc{\cL_1\cL_2}.$$
\end{pth}
\begin{proof} Assume without loss of generality that $k$ is $1$. Suppose, by contradiction, that 
$\sigma$  is a substitution over $\mc{\cL_1\cL_2}$ such that
$$\dermc\sigma(\alpha_1)\quad\ldots\quad\dermc\sigma(\alpha_m).$$
Then, by Proposition~\ref{prop:reflectdersub}, 
$$\der_1\suball{\sigma(\alpha_1)}{1}\quad\ldots\quad\der_1\suball{\sigma(\alpha_m)}{1}.$$
Let $\sigma_1$ be a substitution over $L_1(\Xi)$ such that $\sigma_1(\xi)=\suball{\sigma(\xi)}{1}$ for every schema variable $\xi$. Then,
$$\der_1\sigma_1(\suball{\alpha_1}{1})\quad\ldots\quad\der_1\sigma_1(\suball{\alpha_m}{1})$$
by Proposition~\ref{prop:subsprojs}, and, so,
$$\der_1\lfalsum_1$$
which is a contradiction. Therefore, the admissibility of the rule $\alpha_1\ldots\alpha_m\; / \;\beta$ holds vacuously. 
\end{proof}

\subsubsection*{Meet-combination of $\IPL$ and $\SQT$, and of $\GL$ and $\SQT$}

We now illustrate the results above concerning preservation of admissibility by the meet-combination. 

The Harrop rule 
$$({\lneg}_\IPL \xi_1) \limp_\IPL (\xi_2 \ldisj_\IPL \xi_3) \; / \; ({\lneg}_\IPL \xi_1 \limp_\IPL \xi_2) \ldisj_\IPL ({\lneg}_\IPL \xi_1 \limp_\IPL \xi_3)$$
is admissible in $\IPL$ and so by Corollary~\ref{cor:preseradmiss}, it is also admissible in $\mc{\IPL\,\SQT}$. 

Moreover, the rule
$$({\lposs}_\GL \xi) \lconj_\GL ({\lposs}_\GL {\lneg}_\GL \xi) \; / \; {\lfalsum}_\GL$$
is derivable and so is admissible in $\GL$. Since
$$(\mathop{{\lposs}_{\SQT}} \xi) \mathbin{\lconj_{\SQT}} (\mathop{{\lposs}_{\SQT}}\mathop{{\lneg}_{\SQT}} \xi) \; / \; {\lfalsum}_{\SQT}$$
is admissible in $\SQT$ then, by Theorem~\ref{th:preseradmissibambososlados}, the rule
%

$$
\left(
\bmcu{\lposs_\GL}{\lposs_\SQT} \xi
\right)
\bmcb{\lconj_\GL}{\lconj_\SQT}
\left(
\bmcu{\lposs_\GL}{\lposs_\SQT}
\bmcu{\lneg_\GL}{\lneg_\SQT}
\xi\right) \; \mathbin{\bigg /} \; 
\bmcu{\lfalsum_\GL}{\lfalsum_\SQT}
$$
is admissible in $\mc{\GL\,\SQT}$.

Furthermore, since
$$({\lposs}_{\SQT} \xi)\lconj_{\SQT} ({\lposs}_{\SQT}{\lneg}_{\SQT} \xi) \; / \; {\lfalsum}_{\SQT}$$
is admissible in $\SQT$ then, by Theorem~\ref{th:preseradmissfalsum},
the rule
$$\left(\bmcu{{\lneg}_\GL}{\lposs_{\SQT}} \xi\right) \bmcb{\ldisj_\GL} {\lconj_\SQT} \left(\bmcu{\lneg_\GL} {\lposs_\SQT}\bmcu{\lposs_\GL}{\lneg_\SQT} \xi\right) \; \mathbin{\bigg / }\; 
\bmcu{\lverum_\GL} {\lfalsum_\SQT}$$
is admissible in $\mc{\GL\,\SQT}$.

\subsection{Conservativeness}

The conservative preservation results rely on being able to reproduce the structure of a formula of one of the components in a formula of the other component. For this purpose we work with the decomposition tree of a formula as well as with structure preserving embeddings. 

Given a logic $\cL$, by the \emph{decomposition tree} of a formula $\varphi$ in $L(\Xi)$, denoted by
$$t(\varphi),$$
we mean a rooted tree having as vertices the subformulas of $\varphi$ and as edges the pairs $\langle c(\varphi_1,\ldots,\varphi_m),\varphi_i\rangle$ for each subformula $c(\varphi_1,\ldots,\varphi_m)$ of $\varphi$ and $i=1,\ldots,m$. The source of edge $\langle c(\varphi_1,\ldots,\varphi_m),\varphi_i\rangle$ is $c(\varphi_1,\ldots,\varphi_m)$ and the target is $\varphi_i$.

An \emph{embedding} $h$ of a tree $t_1$ into a tree $t_2$, represented by
$t_1\hookrightarrow_h t_2,$
is a pair composed of an injective map $h_v$ from the vertices of $t_1$ into the vertices of $t_2$ and an injective map $h_e$ from the edges of $t_1$ into the edges of $t_2$ such that for any edge $e_1$ of $t_1$:
$h_v(\emph{source}(e_1))=\emph{source}(h_e(e_1))$, 
$h_v(\emph{target}(e_1))=\emph{target}(h_e(e_1))$ and $\emph{outdegree}(\emph{source}(e_1))=\emph{outdegree}(\emph{source}(h_e(e_1)))$.
We use
$$t(\varphi_1)\approx t(\varphi_2)$$
for asserting that there is an embedding of $t(\varphi_1)$ into $t(\varphi_2)$ and an embedding of $t(\varphi_2)$ into $t(\varphi_1)$.

As we already pointed out we need additional conditions on the component logics in order to prove reflection results. Namely, we assume that the component logics {\it have equivalence}. That is, they have a binary constructor $\leqv$ such that $\der \varphi \leqv \varphi$, 
$\der \varphi \leqv \varphi' $ implies $\der \varphi' \leqv \varphi$, $\der \varphi \leqv \varphi'$ and $\der \varphi' \leqv \varphi''$ implies 
$\der \varphi \leqv \varphi''$, 
$$\der \varphi_k \leqv \varphi'_k, \text{ for  } k=1,\dots,n \quad \Rightarrow \quad \der c(\varphi_1,\dots,\varphi_n) \leqv 
c(\varphi'_1,\dots,\varphi'_n)$$ for each $n$-ary constructor $c$, and $$\der \varphi \text{ and } \der \varphi \leqv \varphi' \quad \Rightarrow \quad \der \varphi'.$$

A logic $\cL$ with equivalence is said to have a constructor $c$ of arity $n$ {\it with identities} for position $k$ if there are $o_1,\ldots,o_{k-1},o_{k+1},\ldots,o_{n}$ in $\Sigma_0$, such that
$$\der c(o_1,\ldots,o_{k-1},\varphi,o_{k+1},\ldots,o_{n}) \leqv \varphi$$
for every formula $\varphi$, where $n \geq 2$ and $1\leq k\leq n$. Such a constructor $c$ is said to have {\it pairwise formula completion} for position $j$ if
for every formula $\psi$ there is a formula $\delta$ such that
$$\der o_j\leqv \delta\quad\text{and}\quad t(\delta)\approx t(\psi)$$
where $j\in\{1,\ldots,k-1,k+1,\ldots,n\}$.

Consider intuitionistic logic. For instance,  $\lconj$ is a binary constructor with identities for position $1$ taking $o_2$ as $\lverum$. Moreover, $\lconj$ has pairwise formula completion for position $2$ as we prove in Proposition~\ref{prop:ttinintermediatehaspairwiseformulacompletion}. As an example of pairwise formula completion, let $\psi$ be the formula
$$p_1 \limp (\lneg p_2).$$ Then $\lverum \ldisj (\lneg \lfalsum)$ is a completion of $\lverum$ for $\psi$. Indeed $\der \lverum \leqv (\lverum \ldisj (\lneg \lfalsum))$
and $t(\lverum \ldisj (\lneg \lfalsum)) \approx t(p_1 \limp (\lneg p_2))$.

For conservative preservation of admissibility we need each component logic to have a constructor with the above properties but with complementary requirements.

We say that logics $\cL$ and $\cL'$ have \emph{complementary  constructors with identities} whenever 
\begin{itemize}
	\item $\cL$ has a $n$-ary constructor $c$ with identities for position $k$ with pairwise formula completion for position $k'$;
	\item $\cL'$ has a $n$-ary constructor $c'$ with identities for  position $k'$ with pairwise formula completion for position $k$;
\item $k\neq k'$.
\end{itemize}
Finally, we say that $\cL$ and $\cL'$ have \emph{similar signatures} whenever $\Sigma_i=\emptyset$ iff $\Sigma'_i=\emptyset$ for $i\in\nats$.

Putting together the requirements above, we now prove that given a formula in each of the component logics, it is always possible to find equivalent formulas with the same decomposition tree.

\begin{prop}\label{prop:formsequivsemsfl}\em
Let $\cL_1$ and $\cL_2$ be logics with equivalence, similar signatures and complementary constructors with identities. Then, for each pair of formulas $\varphi_1$ in $L_1(\Xi)$ and $\varphi_2$ in $L_2(\Xi)$ there is a pair of formulas  $\varphi'_1$ in $L_1(\Xi)$ and $\varphi'_2$ in $L_2(\Xi)$ such that
\begin{itemize}
	\item $\der_1 \varphi_1\leqv\varphi'_1$;
	\item $\der_2 \varphi_2\leqv\varphi'_2$;
	\item $t(\varphi'_1)\approx t(\varphi'_2)$.
\end{itemize}
\end{prop}
\begin{proof} Let $c_1$ be a $n$-ary constructor in $\Sigma_1$ with identities for position $k_1$ and $c_2$ a $n$-ary constructor in $\Sigma_2$ with identities for position $k_2$ such that $k_1\neq k_2$, $c_1$ has pairwise formula completion for $k_2$ and $c_2$ has pairwise formula completion for $k_1$. Let $\psi_1^{\varphi_2}$ be a formula of $L_1(\Xi)$ and $\psi_2^{\varphi_1}$ a formula of $L_2(\Xi)$ such that
$$t(\psi_1^{\varphi_2})\approx t(\varphi_2)\quad\text{and}\quad t(\psi_2^{\varphi_1})\approx t(\varphi_1)$$
which exist since $\cL_1$ and $\cL_2$ have similar signatures. Then, since $c_1$ has pairwise formula completion for $k_2$, there are  formula $\delta^{\varphi_2}_1$ of $L_1(\Xi)$ and $o^{k_2}_1$ in $(\Sigma_1)_0$ with
$$\der_1 o^{k_2}_1\leqv\delta^{\varphi_2}_1\quad\text{and}\quad t(\delta^{\varphi_2}_1)\approx t(\psi^{\varphi_2}_1)$$
and so with $t(\delta^{\varphi_2}_1)\approx t(\varphi_2)$.  Analogously, since $c_2$ has pairwise formula completion for $k_1$, there are  formula $\delta^{\varphi_1}_2$ of $L_2(\Xi)$ and and $o^{k_1}_2$ in $(\Sigma_2)_0$ with 
$$\der_2 o^{k_1}_2\leqv\delta^{\varphi_1}_2\quad\text{and}\quad t(\delta^{\varphi_1}_2)\approx t(\psi^{\varphi_1}_2)$$
and so with $t(\delta^{\varphi_1}_2)\approx t(\varphi_1)$. Assume without loss of generality that $k_1< k_2$. Thus, let $\varphi'_1$ be the formula
$$c_1(o^1_1,\ldots,o^{k_1-1}_1,\varphi_1,o^{k_1+1}_1,\ldots,o^{k_2-1}_1,\delta^{\varphi_2}_1,o^{k_2+1}_1,\ldots,o^{n}_1)$$
and $\varphi'_2$ be the formula
$$c_2(o^1_2,\ldots,o^{k_1-1}_2,\delta^{\varphi_1}_2,o^{k_1+1}_2,\ldots,o^{k_2-1}_2,\varphi_2,o^{k_2+1}_2,\ldots,o^{n}_2).$$
Then
$$\der_1\varphi_1\leqv \varphi'_1\qquad\text{and}\qquad\der_2\varphi_2\leqv \varphi'_2.$$
Moreover, since $t(\varphi_1)\approx t(\delta^{\varphi_1}_2)$ and $t(\delta^{\varphi_2}_1)\approx t(\varphi_2)$ then
$$t(\varphi'_1)\approx t(\varphi'_2).$$
Hence the pair $\varphi'_1$ and $\varphi'_2$ has the required properties.
\end{proof}

The next theorem states that when a rule is admissible in the logic resulting from the meet-combination, then either both of its projections are admissible in the component logics or otherwise necessarily one of its projections is vacuously admissible in one of the component logics.

\begin{pth} \label{th:refladmvac}\em
Let $\cL_1$ and $\cL_2$ be sound logics with equivalence, similar signatures and complementary constructors with identities. Assume  $\alpha_1\ldots\alpha_m \; / \; \beta$ is an admissible rule of $\mc{\cL_1\cL_2}$ such that $\suball{\alpha_1}{j}\ldots\suball{\alpha_m}{j} \; / \; \suball{\beta}{j}$ is not an admissible rule of $\cL_j$ for some $j\in\{1,2\}$. Then,
$$\suball{\alpha_1}{k}\ldots\suball{\alpha_m}{k} \; / \; \lfalsum_k$$
is an admissible rule of $\cL_k$ for $k$ in $\{1,2\}\setminus\{j\}$. 
\end{pth}
\begin{proof} Assume, without loss of generality, that $j$ is $1$, and, so, let $\sigma_1$ be a substitution over $L_1(\Xi)$ such that
$$\der_1\sigma_1(\suball{\alpha_1}{1}),\ldots,\der_1\sigma_1(\suball{\alpha_m}{1})\quad\text{and}\quad\not\der_1\sigma_1(\suball{\beta}{1}).$$
In order to show that $\suball{\alpha_1}{2}\ldots\suball{\alpha_m}{2} \; / \; \lfalsum_2$ is an admissible rule of $\cL_2$ we now show that there is no substitution such that $\der_2\sigma_2(\suball{\alpha_1}{2}),\ldots, \der_2\sigma_2(\suball{\alpha_m}{2})$.
Suppose by contradiction that there is such a substitution, i.e., a substitution $\sigma_2$ over $L_2(\Xi)$ such that 
$$\der_2\sigma_2(\suball{\alpha_1}{2})\quad\ldots\quad\der_2\sigma_2(\suball{\alpha_m}{2}).$$
Taking into account Proposition~\ref{prop:formsequivsemsfl}, let $\sigma_{1\sigma_2}$ and $\sigma_{2\sigma_1}$ be substitutions over $L_1(\Xi)$ and $L_2(\Xi)$, respectively, such that for each schema variable $\xi$, 
$$t(\sigma_{1\sigma_2}(\xi))\approx t(\sigma_{2\sigma_1}(\xi)),\quad \der_1\sigma_1(\xi)\leqv_1\sigma_{1\sigma_2}(\xi)\quad\text{and}\quad\der_2\sigma_2(\xi)\leqv_2\sigma_{2\sigma_1}(\xi).$$
Then, since $\leqv_1$ is an equivalence connective in $\cL_1$ and $\leqv_2$ is an equivalence connective in $\cL_2$, we have
$$\der_1\sigma_{1\sigma_2}(\suball{\alpha_1}{1})\quad\ldots\quad\der_1\sigma_{1\sigma_2}(\suball{\alpha_m}{1}),$$
and
$$\der_2\sigma_{2\sigma_1}(\suball{\alpha_1}{2})\quad\ldots\quad\der_2\sigma_{2\sigma_1}(\suball{\alpha_m}{2}).$$
Let $\sigma$ be a substitution over $\mc{\cL_1\cL_2}$ such that $\suball{\sigma}{1}$ is $\sigma_{1\sigma_2}$ and $\suball{\sigma}{2}$ is $\sigma_{2\sigma_1}$. Since, for each $j=1,\ldots,m$,
$$\sigma_{1\sigma_2}(\suball{\alpha_j}{1})=\suball{\sigma}{1}(\suball{\alpha_j}{1})=\suball{\sigma(\alpha_j)}{1}\quad\text{and}\quad\sigma_{2\sigma_1}(\suball{\alpha_j}{2})=\suball{\sigma}{2}(\suball{\alpha_j}{2})=\suball{\sigma(\alpha_j)}{2},$$
then
$$\der_1\suball{\sigma(\alpha_j)}{1}\quad\text{and}\quad\der_2\suball{\sigma(\alpha_j)}{2},$$
and, so, 
$$\dermc\suball{\sigma(\alpha_j)}{1}\quad\text{and}\quad\dermc\suball{\sigma(\alpha_j)}{2}$$
since the logic resulting from the meet-combination is an extension of each component logic. Then, by lifting,
$$\dermc\sigma(\alpha_1)\quad\ldots\quad\dermc\sigma(\alpha_m),$$
and, so,
$$\dermc\sigma(\beta).$$ Hence, by Proposition~\ref{prop:reflectdersub},
$$\der_1\suball{\sigma(\beta)}{1},$$
that is,
$$\der_1\sigma_{1\sigma_2}(\suball{\beta}{1}),$$
since $\suball{\sigma(\beta)}{1}=\suball{\sigma}{1}(\suball{\beta}{1})=\sigma_{1\sigma_2}(\suball{\beta}{1})$. Taking into account that $\der_1\sigma_1(\xi)\leqv_1\sigma_{1\sigma_2}(\xi)$ 
for each schema variable $\xi$ and that $\leqv_1$ is an equivalence connective in $\cL_1$ we conclude that
$$\der_1\sigma_1(\suball{\beta}{1}),$$
which contradicts the initial assumption that $\not\der_1\sigma_1(\suball{\beta}{1})$.
\end{proof}

So, based on Theorem~\ref{th:refladmvac},
it is immediate to conclude that conservative preservation holds provided that these logics are rich enough. 

\begin{corollary} \label{cor:adminmeetimpliesadminincomp}\em
Let $\cL_1$ and $\cL_2$ be sound logics with equivalence, similar signatures and complementary constructors with identities. Then, 
either in $\cL_1$ the rule $\suball{\alpha_1}{1}\ldots\suball{\alpha_m}{1} \; / \; \suball{\beta}{1}$ is admissible  or in $\cL_2$ the rule $\suball{\alpha_1}{2}\ldots\suball{\alpha_m}{2} \; / \; \suball{\beta}{2}$ is admissible, whenever  $\alpha_1\ldots\alpha_m \; / \; \beta$ is  an admissible rule of $\mc{\cL_1\cL_2}$.
\end{corollary}

\subsubsection*{Meet-combination of $\IPL$ and $\SQT$, and of $\GL$ and $\SQT$}

The first step to illustrate the results above for the meet-combination is to show that the logics at hand fulfil the requirements imposed by those results. In particular, we now show that $\IPL$, $\SQT$ and $\GL$ are logics with constructors with identities enjoying pairwise formula completion. 

\begin{prop}\label{prop:ttinintermediatehaspairwiseformulacompletion}\em In $\IPL$, $\SQT$ and $\GL$, for every formula $\psi$ there are formulas $\delta^\psi_\lverum$ and $\delta^\psi_\lfalsum$ with
$$t(\delta^\psi_\lverum)\approx t(\psi)\text{ and }\der\lverum\leqv\delta^\psi_\lverum,\qquad\text{and}\qquad t(\delta^\psi_\lfalsum)\approx t(\psi)\text{ and }\der\lfalsum\leqv\delta^\psi_\lfalsum.$$
\end{prop}
\begin{proof} The proof follows by induction on the complexity of $\psi$:\\[1mm]
Base: $\psi$ is in $\Sigma_0\cup\Xi$. Then take $\delta^\psi_\lverum$ to be $\lverum$ and $\delta^\psi_\lfalsum$ to be $\lfalsum$.\\[1mm]
Step. Consider the following cases:\\[1mm]
(i)~$\psi$ is either $\lneg\psi_1$. Let $\delta^{\psi_1}_\lfalsum$ be a formula such that $t(\delta^{\psi_1}_\lfalsum)\approx t(\psi_1)$ and $\der\lfalsum\leqv\delta^{\psi_1}_\lfalsum$, and $\delta^{\psi_1}_\lverum$ be a formula such that $t(\delta^{\psi_1}_\lverum)\approx t(\psi_1)$ and $\der\lverum\leqv\delta^{\psi_1}_\lverum$, which exist by induction hypothesis. Take $\delta^\psi_\lverum$ to be $\lneg\delta^{\psi_1}_\lfalsum$, and $\delta^\psi_\lfalsum$ to be $\lneg\delta^{\psi_1}_\lverum$. Then it is immediate to see that the thesis follows. Observe that
 when $\psi$ is  $\lnec\psi_1$ (only applicable to $\SQT$ and GL) a similar proof can be presented. \\[1mm]
(ii)~$\psi$ is $\psi_1\limp\psi_2$. Let $\delta^{\psi_i}_\lfalsum$ be a formula such that $t(\delta^{\psi_i}_\lfalsum)\approx t(\psi_i)$ and $\der\lfalsum\leqv\delta^{\psi_i}_\lfalsum$, and $\delta^{\psi_i}_\lverum$ be a formula such that $t(\delta^{\psi_i}_\lverum)\approx t(\psi_i)$ and $\der\lverum\leqv\delta^{\psi_i}_\lverum$, for $i=1,2$, which exist by induction hypothesis. Then, it is immediate to see that the thesis follows by taking $\delta^\psi_\lverum$ equal to $\delta^{\psi_1}_\lverum\limp\delta^{\psi_2}_\lverum$ and $\delta^\psi_\lfalsum$ equal to $\delta^{\psi_1}_\lverum\limp\delta^{\psi_2}_\lfalsum$.\\[1mm]
(iii)~$\psi$ is either $\psi_1\lconj\psi_2$, $\psi_1\ldisj\psi_2$ or $\psi_1\leqv\psi_2$. We omit the proof of theses cases since it is very similar to the proof of case (ii).
\end{proof}

Using the previous proposition it is immediate to establish the following result.

\begin{prop}\label{prop:cplconjid1compl2}\em In $\IPL$, $\SQT$ and $\GL$,
\begin{itemize}
	\item $\lconj$ and $\ldisj$ are binary constructors with identities for position $1$ and pairwise formula completion for position $2$. Moreover, they also have identities for position $2$ and pairwise formula completion for position $1$; 
	\item $\limp$ is a binary constructor with identities for position $2$ and pairwise formula completion for position $1$. 
\end{itemize}
\end{prop}

We now provide illustrations of the results above concerning the conservative preservation of admissible rules. The example is for the meet-combination of $\SQT$ and $\GL$. Since no substitution over $\mc{\SQT\, \GL}$ makes the premise of the rule
$$
\left(\bmcu{\lnec_\SQT}{\lneg_\GL}
{\xi} \right)
\bmcb{\limp_\SQT}{\lconj_\GL}
\xi  \; \mathbin{\bigg / }\; \xi
$$
a theorem, it is immediate to see that this rule is admissible in $\mc{\SQT\, \GL}$. 
Hence, by Corollary~\ref{cor:adminmeetimpliesadminincomp}, the projection of this rule
into $\GL$ is admissible because the projection of the rule into $\SQT$ is not admissible. Furthermore,
by Theorem~\ref{th:refladmvac}, the rule of $\GL$ with the projection of the premises of the rule above and having as conclusion  $\lfalsum_\GL$, is admissible.

\section{Bases}\label{sec:basespres}

We now concentrate on obtaining a basis for the logic resulting from the meet-combination given bases for the components. We show that a basis can be obtained by the union of the tagged versions of the component bases. 
\begin{figure}
$$
\begin{array}{clr}
1 & \qquad \alpha_1                  	                    & \HYP\\
  & \qquad \ \vdots								& \\
m & \qquad \alpha_m                  	                    & \HYP\\
m+1 & \qquad \suball{\alpha_1}{1}                  	        & \cLFT \;1\\
  & \qquad \ \vdots								& \\
2m & \qquad  \suball{\alpha_m}{1}                  	                    & \cLFT \;m\\
  & \qquad \ \vdots								& \overline{\mathcal{B}}_1\; m+1,\ldots,2m\\
n_1&\qquad \suball{\beta}{1}			  					& \\
n_1+1 & \qquad \suball{\alpha_1}{2}                     			& \cLFT \;1\\
  & \qquad \ \vdots								& \\
n_1+m & \qquad \suball{\alpha_m}{2}                     			& \cLFT \;m\\
  & \qquad \ \vdots		  					& \overline{\mathcal{B}}_2\; n_1+1,\ldots,n_1+m\\
n_2&\qquad \suball{\beta}{2}	\\
n_2+1& \qquad \beta \qquad 						       		& \LFT\; n_1,n_2\\
\end{array}
$$
\caption{Derivation for $\alpha_1\ldots\alpha_m\dermc^{\overline{\mathcal{B}}_1\cup\overline{\mathcal{B}}_2}\beta$ when $\suball{\alpha_1}{1}\ldots\suball{\alpha_m}{1}\; / \; \suball{\beta}{1}$ and $\suball{\alpha_1}{2}\ldots\suball{\alpha_m}{2}\; / \; \suball{\beta}{2}$ are admissible.}
\label{fig:derbasecompsadms}
\end{figure}

\begin{prop}\label{prop:admissimpliesderivnabase}\em
Let $\cL_1$ and $\cL_2$ be sound logics with equivalence, similar signatures, and complementary constructors with identities. Assume that $\mathcal{B}_1$ and $\mathcal{B}_2$ are bases for the admissible rules of $\cL_1$ and $\cL_2$, respectively. Then, 
$$\text{$\alpha_1\ldots\alpha_m\;/\;\beta$ is an admissible rule of $\mc{\cL_1\cL_2}$}\qquad\Rightarrow\qquad\alpha_1\ldots\alpha_m\dermc^{\overline{\mathcal{B}}_1 \cup \overline{\mathcal{B}}_2}\beta,$$
for every formulas $\alpha_1,\ldots,\alpha_m$ and $\beta$ of $\mc{\cL_1\cL_2}$.
\end{prop}
\begin{proof} Let $\alpha_1\ldots\alpha_m\;/\;\beta$ be an admissible rule of $\mc{\cL_1\cL_2}$. Then, one of the following two cases hold:\\[1mm]
(1)~$\suball{\alpha_1}{1}\ldots\suball{\alpha_m}{1}\; / \; \suball{\beta}{1}$ and $\suball{\alpha_1}{2}\ldots\suball{\alpha_m}{2}\; / \; \suball{\beta}{2}$ are admissible. Then, the derivation in Figure~\ref{fig:derbasecompsadms} is a derivation for $$\alpha_1,\ldots,\alpha_m\dermc^{\overline{\mathcal{B}}_1\cup\overline{\mathcal{B}}_2}\beta.$$\\[1mm]
(2)~either $\suball{\alpha_1}{1}\ldots\suball{\alpha_m}{1}\; / \; \suball{\beta}{1}$ or $\suball{\alpha_1}{2}\ldots\suball{\alpha_m}{2}\; / \; \suball{\beta}{2}$ is not admissible. Suppose without loss of generality that $\suball{\alpha_1}{1}\ldots\suball{\alpha_m}{1}\; / \; \suball{\beta}{1}$ is not admissible. Then, the derivation in Figure~\ref{fig:derbaseumadascompsnaoadms} is a derivation for $$\alpha_1,\ldots,\alpha_m\dermc^{\overline{\mathcal{B}}_1\cup\overline{\mathcal{B}}_2}\beta$$
taking into account that,
$\suball{\alpha_1}{2}\ldots \suball{\alpha_m}{2}\; / \; \lfalsum_2$
is an admissible rule of $\cL_2$ by Proposition~\ref{th:refladmvac}. 
\end{proof}

\begin{figure}
$$
\begin{array}{clr}
1 & \qquad \alpha                  	                    & \HYP\\
  & \qquad \ \vdots								& \\
m & \qquad \alpha_m                  	                    & \HYP\\
m+1 & \qquad \suball{\alpha_1}{2}                  	        & \cLFT \;1\\
  & \qquad \ \vdots								& \\
2m & \qquad  \suball{\alpha_m}{2}                  	                    & \cLFT \;m\\
  & \qquad \ \vdots								& \overline{\mathcal{B}}_2\; m+1,\ldots,2m\\
n_1&\qquad \lfalsum_2		  					& \\
  & \qquad \ \vdots								&  \cL_2\\
n_2&\qquad \suball{\beta}{2}                      			& \\
n_2+1 & \qquad \lfalsum_1                     			& \FX\;n_2\\
  & \qquad \ \vdots								& \cL_1\\
n_3&\qquad \suball{\beta}{1}			  					& \\
n_3+1&\qquad \beta			  					& \LFT\;n_3,n_2\\
\end{array}
$$
\caption{Derivation for $\alpha_1\ldots\alpha_m\dermc^{\overline{\mathcal{B}}_1\cup\overline{\mathcal{B}}_2}\beta$ when  $\suball{\alpha_1}{1}\ldots\suball{\alpha_m}{1}\; / \; \suball{\beta}{1}$ is not admissible.}
\label{fig:derbaseumadascompsnaoadms}
\end{figure}

The following result shows that nothing more is derivable in the meet-combination using the extra power of the proposed basis than admissible rules.

\begin{prop}\label{prop:derivnabaseimpliesadmiss}\em
Let $\cL_1$ and $\cL_2$ be sound  logics with equivalence, similar signatures, and complementary constructors with identities. Assume that $\mathcal{B}_1$ and $\mathcal{B}_2$ are bases for the admissible rules of $\cL_1$ and $\cL_2$, respectively. Then, 
$$\text{$\alpha_1\ldots\alpha_m\;/\;\beta$ is an admissible rule of $\mc{\cL_1\cL_2}$}\qquad\Leftarrow\qquad\alpha_1\ldots\alpha_m\dermc^{\overline{\mathcal{B}}_1\cup\overline{\mathcal{B}}_2}\beta,$$
for every formulas $\alpha_1,\ldots,\alpha_m$ and $\beta$ of $\mc{\cL_1\cL_2}$.
\end{prop}
\begin{proof}
The proof follows by induction on the length of a derivation $\psi_1\ldots\psi_m$ for $\alpha_1,\ldots,\alpha_m\dermc^{\overline{\mathcal{B}}_1\cup\overline{\mathcal{B}}_2}\beta$.\\[1mm]
Base. Then, one of the following cases hold:\\[1mm]
(a)~$\beta$ is $\alpha_i$ for some $i$ in $\{1,\ldots,m\}$. Then $\alpha_1,\ldots,\alpha_m\;/\;\beta$ is an admissible rule of $\mc{\cL_1\cL_2}$;\\[2mm]
(b)~$\beta$ is $\rho(\delta)$ for some axiom rule $/\;\delta$ of $\mc{\cL_1,\cL_2}$ and substitution $\rho$ over $\mc{\cL_1\cL_2}$. Then $\dermc\sigma(\beta)$ for every substitution $\sigma$ over $\mc{\cL_1\cL_2}$ since $\sigma(\beta)$ is an instance of $\delta$ by substitution $\sigma\circ\rho$. So $\alpha_1,\ldots,\alpha_m\;/\;\beta$ is an admissible rule of $\mc{\cL_1\cL_2}$;\\[2mm]
(c)~$\beta$ is $\rho(c_1(\delta_1,\ldots,\delta_n))$ where $\;/\;c_1(\delta_1,\ldots,\delta_n)$ is in $\overline{\mathcal{B}}_1$  and $\rho$ is a substitution over $\mc{\cL_1\cL_2}$. Let $\sigma$ be a substitution over $\mc{\cL_1\cL_2}$. Observe that $\der_1\suball{\sigma(\beta)}{1}$ since $\suball{\sigma(\beta)}{1}$ is $\suball{(\sigma\circ\rho)}{1}(c_1(\delta_1,\ldots,\delta_n))$, and either $/\;c_1(\delta_1,\ldots,\delta_n)$ or $/\;\xi$ is a rule in $\mathcal{B}_1$, and, so, is admissible. Observe also that $\der_2\suball{\sigma(\beta)}{2}$ 
since the head constructor of $\suball{\sigma(\beta)}{2}$ is of the form $\mverum{n}_2$ for some $n$. Hence $\dermc\sigma(\beta)$ by lifting, and, so, $\alpha_1,\ldots,\alpha_m\;/\;\beta$ is an admissible rule of $\mc{\cL_1\cL_2}$;\\[2mm]
(d)~$\beta$ is $\rho(c_2(\delta_1,\ldots,\delta_n))$ where $\;/\;c_2(\delta_1,\ldots,\delta_n)$ is in $\overline{\mathcal{B}}_2$  and $\rho$ is a substitution over $\mc{\cL_1\cL_2}$. We omit the proof of this case since it is similar to the proof of case (c).\\[2mm]
Step. Let $r$ be a rule $\delta_1\ldots\delta_k\;/\;\delta$ either in $\mc{\cL_1\cL_2}$ or in $\overline{\mathcal{B}}_1\cup\overline{\mathcal{B}}_2$, $\rho$ a substitution, and  $i_1,\ldots,i_k$ natural numbers with $1\leq i_1,\ldots,i_k<m$, such that $\psi_{i_1}\ldots\psi_{i_k}\;/\;\psi_m$ is an instance of $r$ by $\rho$. So, $\beta$, i.e.~$\psi_m$, is $\rho(\delta)$. Observe that $\alpha_1\ldots\alpha_m\;/\;\psi_{i_j}$ is an admissible rule of $\mc{\cL_1\cL_2}$ by induction hypothesis for $j=1,\dots,k$. Let $\sigma$ be a substitution over $\mc{\cL_1\cL_2}$ such that
$$\dermc\sigma(\alpha_1)\ldots\dermc\sigma(\alpha_m).$$
Then $\dermc\sigma(\psi_{i_j})$ for $j=1,\ldots,k$ by the induction hypothesis. Observe that $\sigma(\psi_{i_1})\ldots\sigma(\psi_{i_k})\;/\;\sigma(\psi_m)$ is an instance of $r$ by substitution $\sigma\circ\rho$. Consider the following cases:\\[1mm]
(a)~$r$ is a rule of $\mc{\cL_1\cL_2}$. Then, $\dermc\sigma(\psi_m)$, i.e., $\dermc\sigma(\beta)$;\\[1mm]
(b)~$r$ is either a non-liberal rule of $\mathcal{B}_1$ or the tagging of a liberal rule of $\mathcal{B}_1$. Hence the head constructor of $\beta$ is from $\cL_1$. Observe that $\der_1\suball{\sigma(\psi_{i_j})}{1}$, i.e.,
$\der_1\suball{(\sigma\circ\rho)}{1}(\delta_j)$
for $j=1,\ldots,k$, by Proposition~\ref{prop:reflectdersub}. So, $\der_1\suball{(\sigma\circ\rho)}{1}(\delta)$, i.e.,
$$\der_1\suball{\sigma(\beta)}{1},$$
since $r$ is also an admissible rule of $\cL_1$. Observe also that $\der_2\suball{\sigma(\beta)}{2}$ 
since the head constructor of $\suball{\sigma(\beta)}{2}$ is of the form $\mverum{n}_2$ for some $n$. Hence
$$\dermc\sigma(\beta)$$
by lifting, and, so, $\alpha_1\ldots\alpha_m\;/\;\beta$ is an admissible rule of $\mc{\cL_1\cL_2}$;\\[2mm]
(c)~$r$ is either a non-liberal rule of $\mathcal{B}_2$ or the tagging of a liberal rule of $\mathcal{B}_2$. We omit the proof of this case since it is similar to the proof of case (b).
\end{proof}

So, taking into account Proposition~\ref{prop:admissimpliesderivnabase} and Proposition~\ref{prop:derivnabaseimpliesadmiss}, we can conclude that the union of the tagged bases of the components is a basis for rule admissibility in the logic resulting from the meet-combination.

\begin{pth}\label{pth:baseadmissmeet}\em
Let $\cL_1$ and $\cL_2$ be sound logics with equivalence, similar signatures, and complementary constructors with identities. Assume that $\mathcal{B}_1$ and $\mathcal{B}_2$ are bases for admissible rules of $\cL_1$ and $\cL_2$, respectively. Then, 
$$\overline{\mathcal{B}}_1\cup\overline{\mathcal{B}}_2$$
is a basis for admissible rules of $\mc{\cL_1\cL_2}$.
\end{pth}

\subsubsection*{Meet-combination of $\IPL$ and $\SQT$}

According to Theorem~\ref{pth:baseadmissmeet}, a basis for the logic resulting from the meet-combi\-nation of $\IPL$ and $\SQT$ is composed by the rules:
$$\left(\bigwedge_{i=1} ^n(\xi_i \limp \xi'_i) \limp \xi_{n+1} \ldisj \xi_{n+2}\right) \ldisj \xi''\; / \; 
\bigvee_{j=1}^{n+2} \left(\bigwedge_{i=1} ^n(\xi_i \limp \xi'_i) \limp \xi_{j}\right) \ldisj \xi''$$
for $n=1,\dots$, and by 
$$(\lposs \xi)\lconj (\lposs \lneg \xi) \; / \; \lfalsum,$$
which are bases for $\IPL$ and $\SQT$, respectively.

\section{Structural completeness, decidability, complexity}\label{sec:structcomp}

Herein, we start by investigating the preservation of structural completeness. Afterwards we address the problem of finding an algorithm for deciding admissibility of rules in the combined logic given algorithms for the components. Finally, we discuss the time complexity of the algorithm.

\begin{figure}
$$
\begin{array}{clr}
1 & \qquad \alpha_1                  	                    & \HYP\\
  & \qquad \ \vdots								& \\
m & \qquad \alpha_m                  	                    & \HYP\\
m+1 & \qquad \suball{\alpha_1}{1}                  	        & \cLFT \;1\\
  & \qquad \ \vdots								& \\
2m & \qquad  \suball{\alpha_m}{1}                  	                    & \cLFT \;m\\
  & \qquad \ \vdots								& \cL_1\; m+1,\ldots,2m\\
n_1&\qquad \suball{\beta}{1}			  					& \\
n_1+1 & \qquad \suball{\alpha_1}{2}                     			& \cLFT \;1\\
  & \qquad \ \vdots								& \\
n_1+m & \qquad \suball{\alpha_m}{2}                     			& \cLFT \;m\\
  & \qquad \ \vdots		  					& \cL_2\; n_1+1,\ldots,n_1+m\\
n_2&\qquad \suball{\beta}{2}	\\
n_2+1& \qquad \beta \qquad 						       		& \LFT\; n_1,n_2\\
\end{array}
$$
\caption{Derivation for $\alpha_1\ldots\alpha_m\dermc\beta$ when $\suball{\alpha_1}{1}\ldots\suball{\alpha_m}{1}\; / \; \suball{\beta}{1}$ and $\suball{\alpha_1}{2}\ldots\suball{\alpha_m}{2}\; / \; \suball{\beta}{2}$ are admissible, and $\cL_1$ and $\cL_2$ are structural complete.}
\label{fig:derstructcomplcompsadms}
\end{figure}

\begin{prop}\label{prop:presstructcompl}\em Let $\cL_1$ and $\cL_2$ be sound logics with equivalence, similar signatures and complementary constructors with identities. Assume that $\cL_1$ and $\cL_2$ are structural complete logics. Then, $\mc{\cL_1\cL_2}$ is structural complete.
\end{prop}
\begin{proof}
Let $\alpha_1\ldots\alpha_m\;/\;\beta$ be an admissible rule of $\mc{\cL_1\cL_2}$. Then, one of the following two cases hold:\\[1mm]
(1)~$\suball{\alpha_1}{1}\ldots\suball{\alpha_m}{1}\; / \; \suball{\beta}{1}$ and $\suball{\alpha_1}{2}\ldots\suball{\alpha_m}{2}\; / \; \suball{\beta}{2}$ are admissible. Then, the derivation in Figure~\ref{fig:derstructcomplcompsadms} is a derivation for $\alpha_1,\ldots,\alpha_m\dermc\beta.$\\[1mm]
(2)~either $\suball{\alpha_1}{1}\ldots\suball{\alpha_m}{1}\; / \; \suball{\beta}{1}$ or $\suball{\alpha_1}{2}\ldots\suball{\alpha_m}{2}\; / \; \suball{\beta}{2}$ is not admissible. Suppose without loss of generality that $\suball{\alpha_1}{1}\ldots\suball{\alpha_m}{1}\; / \; \suball{\beta}{1}$ is not admissible. Then, the derivation in Figure~\ref{fig:derstructcomplumadascompsnaoadms} is a derivation for $\alpha_1,\ldots,\alpha_m\dermc\beta$
taking into account that,
$\suball{\alpha_1}{2}\ldots \suball{\alpha_m}{2}\; / \; \lfalsum_2$
is an admissible rule of $\cL_2$ by Theorem~\ref{th:refladmvac}.
\end{proof}

The previous result allow us to conclude that the logic resulting from the meet-combination of propositional logic and G\"odel-Dummett logic is structurally complete since those logics are structurally complete~\cite{ryb:97}. 

We now investigate the decidability and time complexity of the admissibility problem in the logic resulting from a meet-combination. 
In the sequel, we denote by $\mathbf{P}$ the class of problems decided in polynomial time. 

\begin{figure}
$$
\begin{array}{clr}
1 & \qquad \alpha                  	                    & \HYP\\
  & \qquad \ \vdots								& \\
m & \qquad \alpha_m                  	                    & \HYP\\
m+1 & \qquad \suball{\alpha_1}{2}                  	        & \cLFT \;1\\
  & \qquad \ \vdots								& \\
2m & \qquad  \suball{\alpha_m}{2}                  	                    & \cLFT \;m\\
  & \qquad \ \vdots								& \cL_2\; m+1,\ldots,2m\\
n_1&\qquad 		 \lfalsum_2	  					& \\
  & \qquad \ \vdots								& \cL_2\\
n_2&\qquad \suball{\beta}{2}                      			& \\
n_2+1 & \qquad \lfalsum_1                     			& \FX\;n_1\\
  & \qquad \ \vdots								& \cL_1\\
n_3&\qquad \suball{\beta}{1}			  					& \\
n_3+1&\qquad \beta			  					& \LFT\;n_3,n_2\\
\end{array}
$$
\caption{Derivation for $\alpha_1\ldots\alpha_m\dermc\beta$ when either $\suball{\alpha_1}{1}\ldots\suball{\alpha_m}{1}\; / \; \suball{\beta}{1}$ or $\suball{\alpha_1}{2}\ldots\suball{\alpha_m}{2}\; / \; \suball{\beta}{2}$ is not admissible, and $\cL_1$ and $\cL_2$ are structural complete.}
\label{fig:derstructcomplumadascompsnaoadms}
\end{figure}

\begin{pth}\label{pth:presstructcompl}\em Let $\cL_1$ and $\cL_2$ be sound logics with equivalence, similar signatures and complementary constructors with identities, and $\mathcal{C}_1$ and $\mathcal{C}_2$ time complexity classes. Assume that $\cL_1$ and $\cL_2$ have a decision algorithm for admissibility  in $\mathcal{C}_1$ and $\mathcal{C}_2$ respectively. Then, $\mc{\cL_1\cL_2}$ has a decision algorithm for admissibility  in the upper bound of $\mathcal{C}_1$, $\mathcal{C}_2$ and $\mathbf{P}$.
\end{pth}
\begin{proof} Let $\ADM_{\cL_i}$ be an algorithm in $\mathcal{C}_i$ for admissibility in $\cL_i$ for $i=1,2$.
Consider the algorithm $\ADM_{\mc{\cL_1\cL_2}}$ in Figure~\ref{fig:decisionalgorithmadmissibility}.\\[1mm]
(1)~We start by showing that $\ADM_{\mc{\cL_1\cL_2}}(\{\alpha_1,\ldots,\alpha_m\},\beta)=1$ implies that $\alpha_1\ldots\alpha_m\;/\;\beta$ is an admissible rule of $\mc{\cL_1\cL_2}$. Consider the cases for which the algorithm returns $1$:\\[1mm]
(i)~$\ADM_{\cL_1}(\{\suball{\alpha_1}{1},\ldots,\suball{\alpha_m}{1}\},\suball{\beta}{1})=1$ and $\ADM_{\cL_2}(\{\suball{\alpha_1}{2},\ldots,\suball{\alpha_m}{2}\},\suball{\beta}{2})=1$.
Then,
$\suball{\alpha_1}{1},\ldots,\suball{\alpha_m}{1}\;/\;\suball{\beta}{1}$ is an admissible rule of $\cL_1$  and $\suball{\alpha_1}{2},\ldots,\suball{\alpha_m}{2}\;/\;\suball{\beta}{2}$ is an admissible rule of $\cL_2$. Hence, by Theorem~\ref{th:preseradmissibambososlados}, $\alpha_1\ldots\alpha_m\;/\;\beta$ is an admissible rule of $\mc{\cL_1\cL_2}$ as we wanted to show;\\[1mm]
(ii)~$\ADM_{\cL_1}(\{\suball{\alpha_1}{1},\ldots,\suball{\alpha_m}{1}\},\suball{\beta}{1})=1$ and $\ADM_{\cL_1}(\{\suball{\alpha_1}{1},\ldots,\suball{\alpha_m}{1}\},\lfalsum_1)=1$ and $\ADM_{\cL_2}(\{\suball{\alpha_1}{2},\ldots,\suball{\alpha_m}{2}\},\suball{\beta}{2})=0$. Then $\suball{\alpha_1}{1},\ldots,\suball{\alpha_m}{1}\;/\;\lfalsum_1$ is an admissible rule of $\cL_1$, and, so, by Theorem~\ref{th:preseradmissfalsum}, $\alpha_1\ldots\alpha_m\;/\;\beta$ is an admissible rule of $\mc{\cL_1\cL_2}$ as we wanted to show;\\[1mm]
(iii)~$\ADM_{\cL_1}(\{\suball{\alpha_1}{1},\ldots,\suball{\alpha_m}{1}\},\suball{\beta}{1})=0$ and $\ADM_{\cL_2}(\{\suball{\alpha_1}{2},\ldots,\suball{\alpha_m}{2}\},\suball{\beta}{2})=1$ and $\ADM_{\cL_2}(\{\suball{\alpha_1}{2},\ldots,\suball{\alpha_m}{2}\},\lfalsum_2)=1$. Then $\suball{\alpha_1}{2},\ldots,\suball{\alpha_m}{2}\;/\;\lfalsum_2$ is an admissible rule of $\cL_2$. So, by Theorem~\ref{th:preseradmissfalsum}, $\alpha_1\ldots\alpha_m\;/\;\beta$ is an admissible rule of $\mc{\cL_1\cL_2}$ as we wanted to show.\\[2mm]
(2)~We now show that $\ADM_{\mc{\cL_1\cL_2}}(\{\alpha_1,\ldots,\alpha_m\},\beta)=0$ implies $\alpha_1\ldots\alpha_m\;/\;\beta$ is not an admissible rule of $\mc{\cL_1\cL_2}$. Consider the cases for which the algorithm returns $0$:\\[1mm]
(i)~$\ADM_{\cL_1}(\{\suball{\alpha_1}{1},\ldots,\suball{\alpha_m}{1}\},\suball{\beta}{1})=0$ and $\ADM_{\cL_2}(\{\suball{\alpha_1}{2},\ldots,\suball{\alpha_m}{2}\},\suball{\beta}{2})=0$.
Then, both $\suball{\alpha_1}{1},\ldots,\suball{\alpha_m}{1}\;/\;\suball{\beta}{1}$ and $\suball{\alpha_1}{2},\ldots,\suball{\alpha_m}{2}\;/\;\suball{\beta}{2}$ are not admissible rules of $\cL_1$ and $\cL_2$ respectively. Hence, by Corollary~\ref{cor:adminmeetimpliesadminincomp}, $\alpha_1\ldots\alpha_m\;/\;\beta$ is not an admissible rule of $\mc{\cL_1\cL_2}$, as we wanted to show;\\[1mm]
(ii)~$\ADM_{\cL_1}(\{\suball{\alpha_1}{1},\ldots,\suball{\alpha_m}{1}\},\suball{\beta}{1})=1$ and $\ADM_{\cL_1}(\{\suball{\alpha_1}{1},\ldots,\suball{\alpha_m}{1}\},\lfalsum_1)=0$ and $\ADM_{\cL_2}(\{\suball{\alpha_1}{2},\ldots,\suball{\alpha_m}{2}\},\suball{\beta}{2})=0$. Then $\suball{\alpha_1}{1},\ldots,\suball{\alpha_m}{1}\;/\;\suball{\beta}{1}$ is an admissible rule of $\cL_1$ and $\suball{\alpha_1}{1},\ldots,\suball{\alpha_m}{1}\;/\;\lfalsum_1$ and $\suball{\alpha_1}{2},\ldots,\suball{\alpha_m}{2}\;/\;\suball{\beta}{2}$ are not admissible rules of $\cL_1$ and $\cL_2$ respectively. Suppose by contradiction that $\alpha_1\ldots\alpha_m\;/\;\beta$ is an admissible rule of $\mc{\cL_1\cL_2}$. Then, by Theorem~\ref{th:refladmvac}, $\suball{\alpha_1}{1},\ldots,\suball{\alpha_m}{1}\;/\;\lfalsum_1$ is an admissible rule of $\cL_1$. Contradiction. So, $\alpha_1\ldots\alpha_m\;/\;\beta$ is not an admissible rule of $\mc{\cL_1\cL_2}$;\\[1mm]
(iii)~$\ADM_{\cL_1}(\{\suball{\alpha_1}{1},\ldots,\suball{\alpha_m}{1}\},\suball{\beta}{1})=0$ and $\ADM_{\cL_2}(\{\suball{\alpha_1}{2},\ldots,\suball{\alpha_m}{2}\},\suball{\beta}{2})=1$ and $\ADM_{\cL_2}(\{\suball{\alpha_1}{2},\ldots,\suball{\alpha_m}{2}\},\lfalsum_2)=0$. Then $\suball{\alpha_1}{2},\ldots,\suball{\alpha_m}{2}\;/\;\suball{\beta}{2}$ is an admissible rule of $\cL_2$ and $\suball{\alpha_1}{1},\ldots,\suball{\alpha_m}{1}\;/\;\suball{\beta}{1}$ and $\suball{\alpha_1}{2},\ldots,\suball{\alpha_m}{2}\;/\;\lfalsum_2$ are not admissible rules of $\cL_1$ and $\cL_2$ respectively. Suppose by contradiction that $\alpha_1\ldots\alpha_m\;/\;\beta$ is an admissible rule of $\mc{\cL_1\cL_2}$. Then, by Theorem~\ref{th:refladmvac}, $\suball{\alpha_1}{2},\ldots,\suball{\alpha_m}{2}\;/\;\lfalsum_2$ is an admissible rule of $\cL_2$. Contradiction. So, $\alpha_1\ldots\alpha_m\;/\;\beta$ is not an admissible rule of $\mc{\cL_1\cL_2}$.\\[2mm]
(3)~The decision algorithm $\ADM_{\mc{\cL_1\cL_2}}$ is in the upper bound of $\mathcal{C}_1$,  $\mathcal{C}_2$ and $\mathbf{P}$. Indeed: the operations  executed in step $1$ are in the upper bound of $\mathcal{C}_1$ and $\mathbf{P}$; the operations executed in  step $2$ are in the upper bound of $\mathcal{C}_2$ and $\mathbf{P}$;  the operations performed in steps $3$ and $4$ are in $\mathbf{P}$;  the operations in  step $5$ are in the upper bound of $\mathcal{C}_1$ and $\mathbf{P}$;  and the operations  in step $6$ are in $\mathcal{C}_2$. 
\end{proof}

According to Theorem~\ref{pth:presstructcompl}, the problem of checking if a rule is admissible in the logic resulting from 
the meet-combination of $\IPL$ and $\SQT$ is decidable. Moreover, it is co-NEXP-complete since the admissibility problem of
$\IPL$ is in co-NEXP-complete and the admissibility problem of $\SQT$ is in co-NP-complete.

\begin{figure}
	       \hrule  \vspace*{4mm}
	       Input: $\{\alpha_1,\ldots,\alpha_m\}$ contained in $\Lmc(\Xi)$ and $\beta$ in $\Lmc(\Xi)$\\[2mm]
	       Requires: access to the decision algorithms $\ADM_{\cL_1}$ and $\ADM_{\cL_2}$
\begin{enumerate}
  \item $a_1:=\ADM_{\cL_1}(\{\suball{\alpha_1}{1},\ldots,\suball{\alpha_m}{1}\},\suball{\beta}{1})$;
  \item $a_2:=\ADM_{\cL_2}(\{\suball{\alpha_1}{2},\ldots,\suball{\alpha_m}{2}\},\suball{\beta}{2})$;
  \item If $a_1=1$ and $a_2=1$ then Return $1$;
  \item If $a_1=0$ and $a_2=0$ then Return $0$;
  \item If $a_1=1$ then Return $\ADM_{\cL_1}(\{\suball{\alpha_1}{1},\ldots,\suball{\alpha_m}{1}\},\lfalsum_1)$;
  \item Return $\ADM_{\cL_2}(\{\suball{\alpha_1}{2},\ldots,\suball{\alpha_m}{2}\},\lfalsum_2)$.
  \end{enumerate}
       \hrule  \vspace*{4mm}
\caption{Decision algorithm for admissibility in the meet-combination.}
\label{fig:decisionalgorithmadmissibility}
\end{figure}

\section{Concluding remarks}\label{sec:outlook}

The basic results concerning admissibility when meet-combining logics were established in this paper, including conservative preservation of admissible rules, construction of a basis for the resulting logic from bases given for the component logics, preservation of structural completeness, and preservation of decidability of admissibility with no penalty in the complexity. Meet-combination turned out to be quite well behaved in this respect. 

Taking into account the work in~\cite{ghi:00,ghi:99}, we intend to investigate finitary results for unification in the context of the meet-combination.
Another challenging issue is to extend to meet-combination the algebraic characterisation of admissibility presented in~\cite{met:12}. 

Since the notion of admissible rule does not depend on the particular proof system~\cite{iem:15}, it seems worthwhile to analyse it in the context of consequence systems. This opens the door to studying preservation of admissibility by combination mechanisms at the level of consequence systems.

Beyond meet-combination, we intend to investigate admissibility when using combination mechanisms with a stronger interaction between the components in the resulting logic, namely fusion of modal logics, fibring and modulated fibring, capitalising on the vast work on admissibility in modal logic. Given the stronger interaction we do not expect the results to be as clean as those we obtained for meet-combination.

\section*{Acknowledgments}

This work was partially supported, 
under the PQDR (Probabilistic, Quantum and Differential Reasoning) initiative of SQIG at IT, 
by FCT and EU FEDER, as well as by the project UID/EEA/50008/2013 for the R\&D Unit 50008 financed by the
applicable financial framework (FCT/MEC through
national funds and when applicable co-funded by
FEDER$-$PT2020 partnership agreement), 
and by the European Union's Seventh Framework Programme for Research (FP7) 
namely through project LANDAUER (GA 318287).




\end{document}